\documentclass[11pt,psamsfonts]{amsart}
\usepackage[utf8]{inputenc}

\usepackage{amsmath}
\usepackage{amsthm}
\usepackage{amssymb}
\usepackage{amscd}
\usepackage{amsfonts}
\usepackage{amsbsy}
\usepackage{graphicx}
\usepackage[dvips]{psfrag}
\usepackage{array}
\usepackage{color}
\usepackage{epsfig}
\usepackage{url}
\usepackage[hidelinks]{hyperref}
\usepackage{float}
\usepackage{overpic}
\usepackage{epstopdf}
\usepackage[makeroom]{cancel}
\usepackage{soul}

\newcommand{\R}{\ensuremath{\mathbb{R}}}

\newcommand{\CS}{\ensuremath{\mathcal{S}}}

\newcommand{\CF}{\ensuremath{\mathcal{F}}}

\newcommand{\la}{\lambda}

\newcommand{\f}{\varphi}

\newcommand{\x}{\mathbf{x}}

\newcommand{\sgn}{\mathrm{sign}}

\newcommand{\tang}{\mathrm{tan}}

\newcommand{\vs}{\vspace{0,5cm}}

\def\p{\partial}
\def\e{\varepsilon}

\newtheorem {theorem} {Theorem}
\newtheorem {definition} {Definition}
\newtheorem {proposition} {Proposition}

\newtheorem {example} {Example}
\newtheorem {remark} {Remark}

\begin{document}
\renewcommand{\arraystretch}{1.5}

\title[Sliding mode on tangential sets of Filippov systems]
{Sliding mode on tangential sets\\ of Filippov systems} 
\author[T. Carvalho, D. D. Novaes and D. J. Tonon]
{Tiago Carvalho$^1,$ Douglas D. Novaes$^2$, and Durval J. Tonon$^3$}

\address{$^1$ Departamento de Computaç\~{a}o e Matem\'{a}tica, Faculdade de Filosofia, Ci\^{e}ncias e Letras de Ribeirão Preto,
	USP, Av. Bandeirantes, 3900, CEP 14040-901, Ribeirão Preto, SP, Brazil.}\email{tiagocarvalho@usp.br}

\address{$^2$ Departamento de Matem\'{a}tica - Instituto de Matem\'{a}tica, Estat\'{i}stica e Computa\c{c}\~{a}o Cient\'{i}fica (IMECC) - Universidade
Estadual de Campinas (UNICAMP),  Rua S\'{e}rgio Buarque de Holanda, 651, Cidade Universit\'{a}ria Zeferino Vaz, 13083-859, Campinas, SP,
Brazil} \email{ddnovaes@unicamp.br}

\address{$^3$ Instituto de Matem\'atica e Estat\'istica, Universidade Federal de Goi\'as,\\
	Campus Samambaia, CEP 74690-900, Goi\^ania, GO, Brazil} \email {djtonon@ufg.br}



\maketitle

\begin{abstract}
We consider piecewise smooth vector fields $Z=(Z_+, Z_-)$ defined in $\R^n$ where both vector fields are tangent to the switching manifold $\Sigma$ along a submanifold $M\subset \Sigma$. We shall see that, under suitable assumptions, Filippov convention gives rise to a unique sliding mode on $M$, governed by what we call the {\it tangential sliding vector field}. 
Here, we will provide the necessary and sufficient conditions for characterizing such a vector field. Additionally, we prove that the tangential sliding vector field is conjugated to the reduced dynamics of a singular perturbation problem arising from the Sotomayor-Teixeira regularization of $Z$ around $M$. Finally, we analyze several examples where tangential sliding vector fields can be observed, including a model for intermittent treatment of HIV.
\end{abstract}


\section{Introduction}

In this paper, we consider piecewise smooth vector fields (PSVFs) in two zones defined on $U \subset \mathbb{R}^n$ for which there exists a switching codimension-one manifold $\Sigma$ separating $U$ into two parts $\Sigma_+$ and $\Sigma_-$, such that smooth vector fields $Z_+$ and $Z_-$ are defined in each one of these parts.

Nowadays, there is a vast literature about PSVFs and their applications. A non-exhaustive list of books concerning this theme includes \cite{Brogliato, diBernardo-livro, F, Leine, Simpson, Utkin2009}. Also, papers like \cite{TiagoRonyDurvalLuiz-HIV, Carvalho2020, CarCrisPagTon-PhysicaD-2017, RMCGslidingmode, Jac-To, Kousaka, RMCG2019, TG1, TG2} deal with applications of such theory to real-world phenomena.

Certainly, Filippov's book \cite{F} is one of the most valuable texts concerning PSVFs, where the definition for solutions of discontinuous differential equations is established by means of differential inclusions (see Section \ref{Secao-Preliminares}). Several scenarios of low codimension were deeply investigated by Filippov, which gave rise to the notions of sliding, escaping, and crossing solutions. Along these solutions, the smooth vector fields $Z^+$ and $Z^-$ are transversal to the switching manifold $\Sigma$. Filippov also investigated the dynamics around isolated tangencies between $\Sigma$ and the vector fields $Z^+$ and $Z^-$. However, he has not further explored the differential inclusions when both vector fields $Z_+$ and $Z_-$ are simultaneously tangent to $\Sigma$ along a submanifold $M \subset \Sigma$ of dimension greater or equal to one. Consequently, much of the subsequent development also dealt with the dynamics around isolated tangencies (see, for instance, \cite{CarTeiTon-CuspFoldZAMP, Jeffrey-T-sing, J-T-T2, Novaes2022, Novaes2021, Teixeira1990}), while the dynamics on higher-dimensional tangential manifold remained unexplored.

\subsection{Main goals and results}
In the present paper, we aim to explore the dynamics of Filippov systems on higher-dimensional tangential manifolds.

First, we shall show that, under suitable assumptions, the Filippov convention gives rise to a unique sliding mode on $M$, which is governed by what we call the {\it tangential sliding vector field}. In this direction, Theorem \ref{theo1} provides the necessary and sufficient conditions for the existence of tangential sliding vector fields, which are then formalized in Definition \ref{definicao-campo-tagencial}. Theorem \ref{theo2} establishes that the trajectories of a tangential sliding vector field correspond to solutions of the Filippov differential inclusion.

We shall also investigate how the sliding mode provided by the tangential sliding vector field behaves under Sotomayor-Teixeira regularization. Smoothing processes of PSVF are a good ally in understanding the dynamics and applicability of nonsmooth models. Regarding Filippov systems, Sotomayor-Teixeira regularization, introduced in \cite{Regularizacao}, is intrinsically related to Filippov convention for sliding solutions. Indeed, it was proven in \cite{TEIXEIRA20121948} that the Sotomayor-Teixeira regularization of a Filippov system around a sliding set produces a singular perturbation problem with reduced dynamics conjugated to the sliding dynamics. Also, a similar relation for sliding dynamics with hidden terms was obtained in \cite{Novaes15} by allowing regularizations with non-monotonic transition functions (see also \cite{silvamezanova}). Since Sotomayor-Teixeira's paper \cite{Regularizacao}, isolated tangencies have also been approached by means of regularizations (see, for instance, \cite{BonLarSea,BoneSea,NovRond}), but, again, higher-dimensional tangential manifolds have not been considered. In this regard, Theorem \ref{teorema-regulariado} shows that the tangential sliding vector field is conjugated to the reduced dynamics of a singular perturbation problem arising from the Sotomayor-Teixeira regularization of $Z$ around $M$.

Finally, we analyze several examples where tangential sliding vector fields can be observed, including an applied Filippov model for intermittent treatment of HIV.

\subsection{Structure of the paper}
Section \ref{Secao-Preliminares} provides the basic concepts and notions needed in this paper concerning Filippov systems, regularization processes, and singular perturbation problems. Section \ref{Secao-Existencia-CampoTangencial} presents our first main results, Theorems \ref{theo1} and \ref{theo2}, regarding the definition of tangential sliding vector fields. Section \ref{sec:examples} provides several examples where tangential sliding vector fields can be observed, including an applied Filippov model for intermittent treatment of HIV. Finally, Section \ref{Secao-CampoTangencialRegularizado} is dedicated to proving our third main result, Theorem \ref{teorema-regulariado}, about the regularization of tangential sliding vector fields.

\section{Basic theory}\label{Secao-Preliminares}
This section is devoted to discuss the basic notions needed in this paper. 

\subsection{Piecewise smooth vector fields and Filippov convention}
Let $h: U \subset \R^n \rightarrow \R$ be a smooth function with $0$ as a regular value. Denote $\Sigma = h^{-1}(0)$, $\Sigma_+ = \{ x \in U \mid h(x) \geq 0 \}$, and $\Sigma_- = \{ x \in U \mid h(x) \leq 0 \}$. Consider $Z_+$ a smooth vector field defined in $\Sigma_+$ and $Z_-$ a smooth vector field defined in $\Sigma_-$. A PSVF $Z=(Z_+,Z_-)$ defined in $U \subset \R^n$ is given by
\begin{equation}\label{sistema-descontinuo}
Z(x)= \left\lbrace \begin{array}{cl}
Z_+(x), & h(x) \geq 0, \\
Z_-(x), & h(x) \leq 0.
\end{array} \right.
\end{equation}
For points in $\Sigma_+ \setminus \Sigma$ and $\Sigma_- \setminus \Sigma$, the local trajectories of \eqref{sistema-descontinuo} are given by $Z_+$ and $Z_-$. For points in $\Sigma$, we get that $Z$ is multi-valued and a non-classical theory must be considered. In fact, the notion of local trajectories for PSVF \eqref{sistema-descontinuo} was stated by Filippov \cite{F} as solutions of the following differential inclusion
\begin{equation}\label{FZ}
\dot p\in\CF_Z(p)=\dfrac{Z_+(p)+Z_-(p)}{2}+\sgn(h(p))\dfrac{Z_+(p)-Z_-(p)}{2},
\end{equation}
where
\[
\sgn(s)=\left\{
\begin{array}{ll}
-1&\text{if}\,\, s<0,\\

[-1,1]&\text{if}\,\, s=0,\\
1&\text{if}\,\, s>0.
\end{array}\right.
\]
A solution of the differential inclusion \eqref{FZ} is defined as an absolutely continuous function $\f(t)$ such that $\f'(t) \in \CF_Z(\f(t))$ for almost every $t$.
This approach is called Filippov convention. A PSVF \eqref{sistema-descontinuo} is called a Filippov system when it is governed by Filippov convention. For more information on differential inclusions, see, for instance, \cite{Smirnov02}.

A class of solutions of the differential inclusion \eqref{FZ} was deeply explored by Filippov in \cite{F} and, due to its wide range of application, continues to receive significant attention. In order to better understand the class of solutions of the differential inclusion \eqref{FZ} explored by Filippov, we must determine the contact of the trajectories of $Z_+$ and $Z_-$ with $\Sigma$. For this purpose, consider the Lie derivatives $Z_+h(p) = \left\langle \nabla h(p), Z_+(p) \right\rangle$ and $Z_+^ih(p) = \left\langle \nabla Z_+^{i-1} h(p), Z_+(p) \right\rangle$ for $i \geq 2$, where $\left\langle . , . \right\rangle$ is the usual inner product and $\nabla h(p)$ denotes the gradient of the function $h$ at $p$. The same for $Z_-$.

We distinguish the following regions on the discontinuity set $\Sigma$:

\begin{itemize}
    \item Crossing Region: $\Sigma^c = \{ p \in \Sigma \mid Z_+h(p) \, Z_-h(p) > 0 \}$. Moreover, we let $\Sigma^{c+} = \{ p \in \Sigma \mid Z_+h(p) > 0, Z_-h(p) > 0 \}$ and $\Sigma^{c-} = \{ p \in \Sigma \mid Z_+h(p) < 0, Z_-h(p) < 0 \}$.
    
    \item Sliding Region: $\Sigma^{s} = \{ p \in \Sigma \mid Z_+h(p) < 0, Z_-h(p) > 0 \}$.
    
    \item Escaping Region: $\Sigma^{e} = \{ p \in \Sigma \mid Z_+h(p) > 0, Z_-h(p) < 0 \}$.
\end{itemize}

A point point $q \in \Sigma$ satisfying $(Z_+h(q))(Z_-h(q)) = 0$ and $Z_{\pm}(q) \neq 0$ is called a {\it tangential singularity} (or also a {\it tangency point}). We denote by $\Sigma^{\tang}$ the set of tangential singularities. According to \cite{V} (see also \cite{castro21}), we say that a smooth vector field $Z_{\pm}$ presents a generic contact of multiplicity $k$ with $\Sigma$ at $p$ if $Z_{\pm}^{r}h(p) = 0$ for $r = 1, \dots, k-1$, and $(Z_{\pm})^{k}h(p) \neq 0$.

Despite the richness of the book \cite{F}, Filippov does not further explore solutions of the differential inclusion \eqref{FZ} that may live on $\Sigma^{\tang}$.

\subsection{Regularization and singular perturbation problem}\label{secao regularizacao}
The concept of $\phi-$re\-gu\-la\-rization of PSVFs was introduced by Sotomayor and Teixeira in \cite{Regularizacao}. It provides a $1$-parameter family of smooth vector fields $Z_{\e}$ such that, for each $\e > 0$, $Z_{\e}$ is equal to $Z_{\pm}$ in all points of $\Sigma_{\pm}$ whose distance to $\Sigma$ is greater than or equal to $\e$. In what follows, we provide the formal definition.

\begin{definition}
A function $\phi:\mathbb{R} \rightarrow \mathbb{R}$ is a said to be a $C^{r}$-transition function, $r\geq 0$, if it is of class $C^{r}$, $\phi(x)=-1$ for $x\leqslant -1$, $\phi(x)=1$ for $x\geqslant 1$, and $\phi'(x)>0$ if $x\in(-1,1)$. The $\phi$-regularization of $Z=(X,Y)$ is the $1$-parameter family $Z_{\e}$ given by
\begin{equation}\label{regularization}
Z_{\e}(q)=\left(\frac{1}{2}+\frac{\phi_{\e}(h(q))}{2}\right)Z_+(q) +\left(\frac{1}{2}-\frac{\phi_{\e}(h(q))}{2}\right) Z_-(q),
\end{equation}
with $q\in U$ and $\phi_{\e}(x)=\phi(x/\e)$, for $\e>0$.
\end{definition}

The differential equation 
 $$\dot{q}=Z_{\e}(q)$$ can be studied in terms of \textit{geometric singular perturbation theory} \cite{fenichel,jones}.

\begin{definition}\label{spproblem} Let $U\subseteq  \R^n$ be an open subset and take $\e\geqslant 0$.
	A singular perturbation problem in $U$ (SP-Problem) is a
	differential system which can be  written as
	\begin{equation}
		\label{fast} x'=dx/d\tau=l(x,y,\e),\quad  y'=dy/d\tau=\e m(x,y,\e)
	\end{equation}with $x\in \R$, $y \in \R^{n-1}$ and $l,m$ smooth in all variables, or equivalently, after the time rescaling $t=\e\tau,$
	\begin{equation}
		\label{slow} \e{\dot x}=\e dx/dt=l(x,y,\e),\quad {\dot y}=dy/dt=
		m(x,y,\e).
	\end{equation}
\end{definition}

The differential system \eqref{fast} is called the \textit{fast system}, and the differential system
\eqref{slow} is called the \textit{slow system} of the SP-problem. Note that
for $\e >0$, the phase portraits of the fast and the slow systems
coincide.

By taking  $\varepsilon =0$ in \eqref{fast}, we get the {\it layer problem}
\begin{equation}
\label{layer} x'=dx/d\tau=l(x,y,0),\quad  y'=0.
\end{equation}
Accordingly, the {\it slow set} $\mathcal{S}$ of the SP-problem is defined as the critical points of \eqref{layer}, that is,
\[
\label{SM} \mathcal{S}=\left\lbrace (x,y):l(x,y,0)=0\right\rbrace.
\]
By taking  $\varepsilon =0$ in \eqref{slow}, we get the {\it reduced problem}
\begin{equation}
\label{reduced} 0=l(x,y,\e),\quad {\dot y}=dy/dt=
m(x,y,0),
\end{equation}
which induces a dynamics on $\mathcal{S}$.

\section{Tangential Sliding Vector Field}\label{Secao-Existencia-CampoTangencial}

Let $M\subset \Sigma^{\tang}$ be described by
\[
M = \eta^{-1}(0),
\]
where $\eta: V\subset \R^n \rightarrow \R^m$, with $m<n$ and $V\subset U$ open, is a smooth function with $0 \in \R^m$ as a regular value. 
The next result is a well-known fact concerning Differential Geometry (see \cite[Corollary 5.24 and Lemma 5.29]{Lee-book}):
\begin{proposition}\label{prop:GD}
If $\eta: \R^n \rightarrow \R^m$, where $m<n$, is a smooth function with $0 \in \R^m$ as a regular value, then $M = \eta^{-1}(0)$ is a codimension $m$ submanifold, and
\[  
T_{p}M = \text{ker}(d\eta(p)) \,\, \forall p \in M,
\]
where $d\eta$ is the differential of $\eta$.
\end{proposition}

In what follows, we shall study necessary and sufficient conditions for the existence of a tangential sliding vector field on $M$. For each $p \in \Sigma$, define 
\[
C_p = \{C(p, \lambda) : \lambda \in [-1, 1]\},
\]
where
\begin{equation}\label{eq-combinacao-convexa}
C(p, \lambda) = \frac{(1 - \lambda)}{2} Z_+(p) + \frac{(1 + \lambda)}{2} Z_-(p).
\end{equation}
For $p \in M$, the tangential sliding vector field will be constructed by means of the intersection of the set $C_p$ with the tangent space of $M$ at $p$, $T_p M$. Accordingly, the next result provides necessary and sufficient conditions for $C_p \cap T_{p} M \neq \emptyset$.

\begin{theorem}\label{theo1}
Consider $p \in M$. Thus, $C_p\cap T_{p}  M \neq\emptyset$ if, and only if,
\begin{equation}\label{fundcond}
\left\langle  d\eta(p)   Z_+(p) ,  d\eta(p)   Z_-(p)  \right\rangle=-\|d\eta(p)   Z_+(p)\| \| d\eta(p)   Z_-(p)\|.
\end{equation}
Moreover,
\begin{itemize}
\item [$1.$] If $\|d\eta(p)   Z_+(p)\| \| d\eta(p)   Z_-(p)\|\neq0$, then $C_p\cap T_{p} M =\{C(p,\la^*(p))\}$ with
\begin{equation}\label{eq lambda star}\displaystyle\lambda^*(p):=\frac{\| d\eta(p)   Z_+(p) \| - \| d\eta(p)   Z_-(p) \|}{\| d\eta(p)   Z_+(p) \| + \| d\eta(p)   Z_-(p) \|}\in(-1,1).\end{equation}

\smallskip

\item [$2.$] If $d\eta(p)  Z_-(p) = 0$ and $d\eta(p)  Z_+(p) \neq 0,$  then $C_p\cap T_{p} M =\{Z_-(p)\}$.

\smallskip

\item [$3.$] If $d\eta(p)  Z_+(p) = 0$ and $d\eta(p)  Z_-(p) \neq 0,$  then $C_p\cap T_{p} M =\{Z_+(p)\}$.

\smallskip

\item [$4.$] If  $d\eta(p)  Z_-(p) =d\eta(p)  Z_+(p)= 0,$ then  $C_p\subset T_{p} M $.
\end{itemize}

\end{theorem}

\begin{proof}
Since, from Proposition \ref{prop:GD}, $T_{p}  M  = \text{ker}(d\eta(p))$, we have
\[
C(p, \lambda) \in T_{p} M \iff d\eta(p) \cdot C(p, \lambda) = 0,
\]
replacing the expression \eqref{eq-combinacao-convexa} we obtain
\begin{equation}\label{eq igualdade vetores}
\frac{(1- \lambda)}{2} d\eta(p) Z_+(p) = -\frac{(1+ \lambda)}{2} d\eta(p) Z_-(p).
\end{equation}
Taking into account that $\lambda \in [-1, 1]$, the relationship \eqref{eq igualdade vetores} means that the vectors $d\eta(p) Z_+(p)$ and $d\eta(p) Z_-(p)$ are proportional with a negative proportionality constant. Hence, $C_p \cap T_{p}  M \neq \emptyset$ if, and only if, \eqref{fundcond} holds.

\bigskip

\noindent {\it Case 1.} Assuming that $\|d\eta(p) Z_+(p)\| \| d\eta(p) Z_-(p)\| \neq 0$, we get that \eqref{eq igualdade vetores} holds if, and only if, $\lambda = \lambda^*(p)$, where $\lambda^*(p)$ is given by \eqref{eq lambda star}. Notice that, in this case, $\lambda^*(p) \in (-1, 1)$.

\bigskip

\noindent {\it Cases 2 and 3.} Assuming that $d\eta(p) Z_-(p) = 0$ and $d\eta(p) Z_+(p) \neq 0$, the relationship \eqref{eq igualdade vetores} reduces to $(1-\lambda) d\eta(p) Z_+(p) = 0$, which holds if, and only if, $\lambda = 1$. In this case, $C_p \cap T_{p} M = \{Z_-(p)\}$. In an analogous way, for Case 3 we conclude that $C_p \cap T_{p} M = \{Z_+(p)\}$.

\bigskip

\noindent {\it Cases 4.} Assuming $d\eta(p) Z_-(p) = d\eta(p) Z_+(p) = 0$, the relationship \eqref{eq igualdade vetores} holds for every $\lambda \in [-1, 1]$. Hence, $C_p \subset T_{p} M$.
\end{proof}

In what follows, considering Theorem \ref{theo1}, we provide the definition of tangential sliding vector fields.

\begin{definition}\label{definicao-campo-tagencial}
Assume that 
\begin{equation}\label{cond}
\left\langle  d\eta(p)   Z_+(p) ,  d\eta(p)   Z_-(p)  \right\rangle=-\|d\eta(p)   Z_+(p)\| \| d\eta(p)   Z_-(p)\|\neq0
\end{equation}
for every $p\in M $. The \textbf{tangential sliding vector field}, $Z^{\tang}_M:M\to TM$, of $Z$ on $M$ is defined by 
\begin{equation}\label{eq-campo-tangencial}
Z^{\tang}_M(p):=C(p,\lambda^*(p))= \frac{ (1- \lambda^*(p))}{2} Z_+(p) +  \frac{ (1 + \lambda^*(p))}{2} Z_-(p),
\end{equation}
where $\lambda^*(p)\in(-1,1)$ is given by \eqref{eq lambda star}.
\end{definition}

\begin{remark}
Definition \ref{definicao-campo-tagencial} provides a tangential sliding vector field $Z^{\tang}_M$ on a submanifold $M \subset \Sigma$ for which the conditions of Case 1 of Theorem \ref{theo1} hold. Such a vector field can be extended for points on the boundary of $M$ for which the conditions of Cases 2 and 3 of Theorem \ref{theo1} hold, but not for Case 4. That is because, under condition \eqref{fundcond} of Theorem \ref{theo1}, $C_p \cap T_p M$ is not single-valued only in Case 4. In addition, Cases 2 and 3 are limiting scenarios of Case 1. 
\end{remark}

Notice that a trajectory $\phi : I \to M$ of the tangential sliding vector field \eqref{eq-campo-tangencial} satisfies
\[
\f'(s)=Z_M^{\tang}(\f(s))=C(\f(s),\lambda^*(\f(s)))\in \CF_Z(\f(s)),
\]
for every $s \in I$. Accordingly, we have proven the following result providing that trajectories of the tangential sliding vector field \eqref{eq-campo-tangencial} constitute, indeed, a class of solutions of the differential inclusion \eqref{FZ} and, therefore, a class of Filippov trajectories of the PSVF \eqref{sistema-descontinuo}:
\begin{theorem}\label{theo2}
Any trajectory of the tangential sliding vector field \eqref{eq-campo-tangencial} is a solution of the differential inclusion \eqref{FZ}.
\end{theorem}

\section{Some examples of tangential sliding vector fields}\label{sec:examples}

The function $\eta$, which describes the manifold $M$ contained in $\Sigma^{\tang}$, and the tangential sliding vector field $Z^{\tang}_M$, depend on the multiplicity of the tangential contact between $Z+$ and $Z_-$ with $\Sigma$ at points of $M$. This section is devoted to exploring several examples where the function $\mathbf{\eta}$ can be obtained in terms of $h$ and the Lie derivatives $Z_+^i h$.

\smallskip

\subsection{Example 1: Multiplicity 2 tangential manifold}\label{Exemplo-Dobra-dobra} In this first example, we investigate the sliding tangential vector field defined on a manifold of  tangential points multiplicity $2$.

Consider the PSVF $Z=(Z_+, Z_-):  \mathbb{R}^4 \rightarrow \mathbb{R}^4$ where 
	\[
	\begin{array}{ll}
		Z_+(x_1,x_2,x_3,x_4)  =&(a_1, a_2, a_3, a_4 x_1),\\
		Z_- (x_1,x_2,x_3,x_4) =&(b_1, b_2, b_3, b_4 x_1),
	\end{array}
	\]
	and the switching manifold is given by $h(x_1, \dots , x_4)=x_4$. Computing the Lie derivatives, we get
	\[
	\begin{array}{ll}
		Z_+h(x_1, x_2, x_3,0)=a_4x_1,   & Z_-h(x_1, x_2, x_3,0)=b_4x_1, \\
		Z_+^2h(x_1, x_2, x_3,0)=a_1a_4, &Z_-^2h(x_1, x_2, x_3,0)=b_1b_4. 
	\end{array}
	\]
	We assume that $a_1\, a_4\, b_1\, b_4\neq 0$.
	So, the manifold $M\subset \Sigma^{\tang} \subset \Sigma \subset \mathbb{R}^4$ that corresponds to the tangential points is given by $M=\eta^{-1}(0)$, where 
	\[
	\eta(x_1, x_2, x_3, x_4)=(x_1, x_4).
	\]
	Notice that $d\eta(p)\, Z_+(p)=(0, a_1)$ and $d\eta(p)\, Z_-(p)=(0, b_1)$. Thus, condition \eqref{cond} of Definition \ref{definicao-campo-tagencial} is satisfied provided that $a_1\, b_1<0$.
	Accordingly, the tangential vector field is given by
\begin{equation}\label{Exemplo-CampoTangencial-DobraR4}
Z^{\tang}_M(x_2, x_3)=\Big(0,\dfrac{a_2 \left| b_1\right| +b_2 \left| a_1\right| }{\left| b_1\right| +\left| a_1\right| }, \frac{a_3 \left| b_1\right| +b_3 \left| a_1\right| }{\left| b_1\right| +\left| a_1\right| },0\Big).
\end{equation}

\subsection{Example 2: Higher order multiplicity tangential manifold}
It is possible for a PSVF, $Z=(Z_+,Z_-),$ to have tangential points with distinct multiplicities in such a way that, for each multiplicity $i$, a manifold $M_i\subset\Sigma^{\tang},$ constituted by the tangential points of multiplicity $i,$ can be defined. In this case, a tangential sliding vector field can be defined on each manifold $M_i$. Accordingly, in the next example we consider a class of PSVFs $Z=(Z_+, Z_-)$ defined in $\R^n$ for which both vector fields $Z_{\pm}$ have a generic contact of multiplicity $m, l\leq n$.

Consider the PSVF $Z=(Z_+, Z_-):  \R^n \rightarrow \R^n$ where 
\[
\begin{array}{ll}
Z_+(\bf{x})  =&(a_1 x_2, a_2 x_3, \dots, a_{m} x_{m+1}, \dots,  a_{l-2}x_{l-1}, a_{l-1}, \dots, a_{n-1}, x_1)\\
Z_-(\bf{x})  =&(b_1 x_2, b_2 x_3, \dots, b_{m-2} x_{m-1}, b_{m-1}, \dots, b_{n-1}, x_1),
\end{array}
\]
with $a_i\, b_j \neq 0, i=1, \dots, l-1, j=1, \dots, m-1$ and the switching manifold given by $h({\bf x})=x_n$. Here, ${\bf x}=(x_1,\ldots,x_n)$. 
Notice that the origin is a contact point of $Z_+$ and  $Z_-$ with $\Sigma$ of multiplicity $l$ and $m$, respectively. Assume $l>m$. Computing the Lie derivatives, we get
\[
\begin{array}{ll}
Z_+h({\bf x})=x_1,   & Z_-h({\bf x})=x_1, \\
Z_+ ^2h({\bf x})=a_1x_2,           &Z_- ^2h({\bf x})=b_1x_2,\\
\vdots & 	\vdots\\
Z_+ ^{m-1}h({\bf x})=a_1\dots a_{m-2}x_{m-1}, & Z_- ^{m-1}h({\bf x})=b_1\dots b_{m-2}x_{m-1}, \\
Z_+ ^{m}h({\bf x})=a_1\dots a_{m-1}x_{m}, & Z_- ^{m}h({\bf x})=b_1\dots b_{m-1}\neq 0,\\
Z_+ ^{m+1}h({\bf x})=a_1\dots a_{m}x_{m+1}, & \\
\vdots & 	\\
Z_+ ^{l-1}h({\bf x})=a_1\dots a_{l-2}x_{l-1}, & \\
Z_+ ^{l}h({\bf x})=a_1\dots a_{l-1}\neq 0. & 	\end{array}
\]
The manifold $M_2\subset \Sigma^{\tang}$, composed of the tangential points of multiplicity $2$ is given by $\eta_{2}^{-1}(0)$, where $\eta_2:\R^n\setminus\{\x:\,x_2=0\}\to\R^2$ is given by
\[
\eta_{2}(x_1, \dots, x_n)=(h({\bf x}), Z_+h({\bf x}))=(x_n, x_1).
\]
Notice that for $p_2\in M_2$, $d\eta_{2}(p_2)\, Z_+(p_2)=(0, a_1 x_2)$ and $d\eta_{2}(p_2)\, Z_-(p)= (0, b_1x_2)$. Thus, condition \eqref{cond} of Definition \ref{definicao-campo-tagencial} is satisfied provided that $a_1\, b_1<0$.
Accordingly, the tangential vector field defined on $M_2$ is given by
\[
\begin{aligned}
Z_{M_2}^{\tang}(p_2)    =& \dfrac{|b_1|}{|a_1|+|b_1|} Z_+(p_2) + \dfrac{|a_1|}{|a_1|+|b_1|} Z_-(p_2)\\\\
=&\dfrac{1}{|a_1|+|b_1|}\Big( 0,x_3(a_2|b_1|+b_2|a_1|), \dots, x_{m-1}(a_{m-2}|b_{1}| + b_{m-2}|a_1|), \\\\
& x_m a_{m-1}|b_1|+b_{m-1}|a_1|, \dots, a_{n-1}|b_1|+b_{n-1}|a_1|,0 \Big).
\end{aligned}
\]
Analogously, the manifold $M_3\subset \Sigma^{\tang}$ composed of the tangential points of multiplicity $3$ is given by $\eta_{3}^{-1}(0)$, where $\eta_3:\R^n\setminus\{\x:\,x_3=0\}\to\R^3$ is given by
\[
\eta_{3}(x_1, \dots, x_n)=(h({\bf x}), Z_+h({\bf x}), Z_+ ^2h({\bf x}))=(x_n, x_1, a_1x_2).
\]
Notice that for $p_3\in M_3$, $d\eta_{3}(p_3)\, Z_+(p_3) = (0, 0, a_1 a_2 x_3)$ and $d\eta_{3}(p_3)\, Z_-(p)= (0,0, a_1b_2 x_3)$. Thus, condition \eqref{cond} of Definition \ref{definicao-campo-tagencial} is satisfied provided that $a_2\, b_2<0$.
Accordingly, the tangential vector field defined on $M_3$ is given by
\[
\begin{aligned}
Z_{M_3}^{\tang}(p_3)   =& \dfrac{|b_2|}{|a_2|+|b_2|} Z_+(p_2) +  \dfrac{|a_2|}{|a_2|+|b_2|} Z_-(p_2)\\\\
=&\dfrac{1}{|a_2|+|b_2|}\Big( 0, 0, x_4(a_3|b_2|+b_3|a_2|), \dots, x_{m-1}(a_{m-2}|b_{2}| + b_{m-2}|a_2|), \\\\
& x_m a_{m-1}|b_2| + b_{m-1}|a_2|, \dots, a_{n-1}|b_2| + b_{n-1}|a_2|,0 \Big).
\end{aligned}
\]
In general, the manifold $M_{m}\subset \Sigma^{\tang}$ composed of the tangential points of multiplicity $m$ is given by $\eta_{m}^{-1}(0)$, where $\eta_{m}:\R^n\setminus\{\x:\,x_{m}\geq 0\}\to\R^m$ is given by
\[
\begin{aligned}
\eta_{m}(x_1, \dots, x_n)  =&(h({\bf x}), Z_+h({\bf x}), Z_+ ^2h({\bf x}), \dots, Z_+ ^{m-1}h({\bf x}))\\
=&(x_n, x_1, a_1x_2, \dots, a_1 \dots a_{m-2}x_{m-1}).
\end{aligned}
\]
Notice that for $p_{m}\in M_{m}$, $d\eta_{m}(p_{m})\, Z_+(p_{m}) = (0, \dots, 0, a_1 a_2 \dots  a_{m-1}x_{m})$ and $d\eta_{m}(p_{m})$ $Z_-(p_{m})= (0,\dots, 0, a_1 a_2 \dots a_{m-2} b_{m-1})$. Thus, condition \eqref{cond} of Definition \ref{definicao-campo-tagencial} is satisfied provided that $a_{m-1}\, b_{m-1}>0$.
Accordingly, the tangential vector field defined on $M_{m}$ is given by
\[
\begin{aligned}
Z_{M_{m}}^{\tang}(p_{m})   =& \dfrac{|b_{m-1}|}{|a_{m-1}x_{m}| + |b_{m-1}|} Z_+(p_m) +  \dfrac{|a_{m-1}x_{m}|}{|a_{m-1}x_{m}| + |b_{m-1}|} Z_-(p_m)\\\\
	=& \dfrac{1}{b_{m-1} -a_{m-1} x_{m}} \Big( 0, 0, \dots, 0, - a_{m-1} x_{m} b_{m} + b_{m-1} a_mx_{m+1}, \dots, \\\\
	&- a_{m-1} x_{m} b_{l-2} + b_{m-1} a_{l-2}x_{l-1}, - a_{m-1} x_{m} b_{l-1} + b_{m-1} a_{l-1}, 0 \Big).
\end{aligned}
\]
Finally, we notice that, although one can define a manifold composed of contact points of multiplicity higher than $m$, a tangential vector field cannot be defined on it. In fact, let $M_{m+1}\subset \Sigma^{\tang}$ be the submanifold composed of the tangential points of multiplicity $m+1$, that is, $M_{m+1}=\eta_{m+1}^{-1}(0)$, where $\eta_{m+1}:\R^n\setminus\{\x:\,x_{m+1}=0\}\to\R^{m+1}$ is given by
\[
\begin{aligned}
	\eta_{m}(x_1, \dots, x_n)  =&(h({\bf x}), Z_+h({\bf x}), Z_+ ^2h({\bf x}), \dots, Z_+ ^{m-1}h({\bf x}), Z_+ ^{m}h({\bf x}) )\\
	=&(x_n, x_1, a_1x_2, \dots, a_1 \dots a_{m-2}x_{m-1}, a_1 \dots a_{m-1}x_{m}).
\end{aligned}
\]
Thus, for $p_{m+1}\in M_{m+1}$, one has $d\eta_{m+1}(p_{m+1})\, Z_+(p_{m+1}) = (0, \dots, 0,$\linebreak $ a_1 a_2 \dots  a_{m}x_{m+1})$ and $d\eta_{m+1}(p_{m+1})\, Z_-(p_{m+1})= (0,\dots, 0, a_1 a_2 \dots a_{m-2} b_{m-1},$ \linebreak $ a_1 a_2 \dots a_{m-1} b_{m})$. Thus, condition \eqref{cond} of Definition \ref{definicao-campo-tagencial} is satisfied provided that  
\[
a_1 a_2 \dots a_{m-2} b_{m-1} =0\quad\text{and}\quad a_{m}\, b_{m}x_{m+1}<0,
\] 
which contradicts the hypothesis $a_i\, b_j \neq 0, i=1, \dots, l-1$ and $j=1, \dots, m-1$ (recall that $l>m$).

\subsection{Example 3: A model of intermittent treatment of HIV}\label{sec aplicacoes}

Consider the PSVF $Z=(Z_+,Z_-):U\subset \R^3 \rightarrow \R^3$ where
\begin{equation}\label{equacao-campo-AIDS}
\begin{array}{ll}
Z_+  =&(-k x_1 x_3 + s-\alpha  x_1, k x_1 x_3 - \delta  x_2, \theta  x_2 - c x_3),\\
Z_-  =&(-(1-\eta_{RT}) k x_1 x_3 + s-\alpha  x_1, (1-\eta _{RT}) k x_1 x_3 - \delta  x_2, \\
&(1-\eta_{PI}) \theta  x_2 -c x_3).
\end{array}
\end{equation}
The PSVf above models the intermittent treatment of human immunodeficiency virus. For more details about the formal mathematical analysis of this model, see \cite{TiagoRonyDurvalLuiz-HIV}. Here, $x_2(t)$ denotes the infected cell population size at time $t$, $x_1(t)$ denotes the uninfected cell population size at time $t$, $x_3(t)$ denotes the concentration of infectious virus particles at time $t$, and the parameters $\eta_{RT}, \eta_{PI}, k, \alpha, \delta, \theta, c$ are positive real numbers related to the properties of the model and the treatment protocol of this disease. For biological reasons, it is assumed that
\begin{equation}\label{bio}
\alpha < \delta \,\, , \,\, k s \theta > c \alpha \delta \,\, \mbox{ and } \,\, \frac{s}{\delta} +  \frac{c (\delta - \alpha)}{k \theta} < C_T < min \Big\{ \frac{s}{\delta}, \frac{s}{\delta} +  \frac{c (\delta - \alpha)}{(\eta_{RT}-1)(\eta_{PI}-1)k \theta}\Big\}. 
\end{equation}

The switching manifold is given by $h(x_1,x_2,x_3)= x_1+x_2-C_T$, where $C_T$ is a positive real number that controls when the antiretroviral therapy will be triggered; more precisely, the therapy is triggered below this value and stopped above this value. The first Lie derivatives are given by
\[
\begin{array}{ll}
Z_+h(C_T-x_2, x_2, x_3)  =&-\alpha  C_T+s+x_2 (\alpha -\delta ),   \\
 Z_-h(C_T-x_2, x_2, x_3)  =&-\alpha  C_T+s+x_2 (\alpha -\delta ).
\end{array}
\]
Solving $Z_{\pm}h=0$ yields the curve
\[
\Sigma^{\tang}=\Big\{\Big(\dfrac{s-C_T \delta }{\alpha -\delta }, \dfrac{\alpha  C_T-s}{\alpha -\delta }, x_3\Big); x_3\geq 0\Big\}
\]
and, over $\Sigma^{\tang},$ there exist two cusp points given by: 
\[
\begin{array}{ll}
p_c^+   =&\Big(\dfrac{s-C_T \delta }{\alpha -\delta }, \dfrac{\alpha  C_T-s}{\alpha -\delta },\dfrac{\delta  (\alpha  C_T-s)}{k (s-C_T \delta )}\Big),\\\\
p_c^-    =&\Big(\dfrac{s-C_T \delta }{\alpha -\delta }, \dfrac{\alpha  C_T-s}{\alpha -\delta }, \dfrac{\delta  (\delta -\alpha ) (\alpha  C_T-s)}{(\eta _{RT}-1) k (\alpha -\delta ) (s- C_T \delta )}\Big).
\end{array}
\]

The submanifold $M_2 \subset \Sigma^{\tang}$ of tangential points of multiplicity $2$ is given by $\eta_2^{-1}(0)$, where  $\eta:\R^3\setminus\{p_c^+ ,p_c^- \}\to\R^2$ is given by 
\[
\eta_2(x_1, x_2, x_3)=(h(x_1, x_2, x_3), Z_+h(C_T-x_2, x_2, x_3)).
\]
Let $p=\Big(\dfrac{s-C_T \delta }{\alpha -\delta }, \dfrac{\alpha  C_T-s}{\alpha -\delta }, x_3\Big)\in\Sigma^{\tang}$. It can be noted that
\[
d\eta_2(p)\, Z_+(p)=\Big(0, k x_3 (s-C_T \delta )+\frac{\delta  (\delta -\alpha ) (\alpha  C_T-s)}{\alpha -\delta }\Big)
\]
and 
\[
d\eta_2(p)\, Z_-(p)=\Big(0, \frac{\delta  (\delta -\alpha ) (\alpha  C_T-s)}{\alpha -\delta }-(\eta_{RT}-1) k x_3 (s-C_T \delta )\Big).
\]
Thus, under assumptions \eqref{bio}, the condition \eqref{cond} of Definition \ref{definicao-campo-tagencial} is satisfied. In this case,
\[ \lambda^*(p) = \frac{-2 \alpha  C_T \delta + C_T \delta \eta_{RT} k x_3 - 2 C_T \delta  k x_3 - \eta_{RT} k s x_3 +2 k s x_3+2 \delta  s}{\eta_{RT} k x_3 (s- C_T \delta )}. \]

The tangential vector field \eqref{eq-campo-tangencial} is given by
\[
\begin{array}{ll} 
Z^{\tang}_{M_2}(x_3)  = &\Big(0,0,-\dfrac{\alpha  c x_3}{\alpha -\delta } + \dfrac{c \delta  x_3}{\alpha -\delta } - \dfrac{\alpha  C_T \eta_{PI} \theta }{\alpha -\delta } + \dfrac{\alpha  C_T \theta }{\alpha -\delta } \\\\
&+ \dfrac{\alpha  C_T \eta_{PI} \theta }{(\alpha -\delta ) \left(1-\frac{\delta  (s-C_T (\alpha +k x_3))+k s x_3}{\delta  (s-\alpha  C_T)-(\eta_{RT}-1) k x_3 (s-C_T \delta )}\right)}\\\\
&-\dfrac{\eta_{PI}\theta  s}{(\alpha -\delta ) \left(1-\frac{\delta  (s-C_T (\alpha +k x_3))+k s x_3}{\delta  (s-\alpha  C_T)-(\eta_{RT}-1) k x_3 (s-C_T \delta )}\right)} + \dfrac{\eta_{PI} \theta  s}{\alpha -\delta }-\dfrac{\theta  s}{\alpha -\delta }\Big).
\end{array}
\]
This expression coincides with the expression of the tangential vector field given in \cite{TiagoRonyDurvalLuiz-HIV}. Note that this tangential vector field possesses an equilibrium point 
\[
\begin{array}{ll}
p_2^*=&\Big(0,0, \frac{1}{2 c \eta_{RT} k (\alpha -\delta ) (C_T \delta -s)}\Big(-k \theta  (\eta_{PI}-\eta_{RT}) (s-\alpha  C_T) (s-C_T \delta ) +\sqrt{\Delta}\Big)\Big),
\end{array}
\]
where \[\Delta=k \theta  (s-\alpha  C_T)^2 (s-C_T \delta ) \left(4 c \delta  \eta_{PI} \eta_{RT} (\alpha -\delta )+k \theta  (\eta_{PI}-\eta_{RT})^2 (s-C_T \delta )\right).\]

\section{Regularized tangential vector field}\label{Secao-CampoTangencialRegularizado}

Let $Z=(Z_+,Z_-)$ be a PSVF. Let $\eta:V\subset U\to\R^m,$ with $m<n$ and $V\subset U$ open, be a smooth function is defined by $0\in\R^m$ as a regular value. Assume that $M = \eta^{-1}(0)\subset\Sigma^{\tang}$ satisfies condition \eqref{cond} of Definition \ref{definicao-campo-tagencial}. Following \cite{Novaes15}, the next result shows that the tangential sliding vector field $Z_{M}^{\tang}$ is locally conjugated to the reduced dynamics of a SP-Problem arisen from the regularization of $Z$, and restricted to a manifold contained in the slow set.

\begin{theorem}\label{teorema-regulariado}
For each $p\in M$, there exists a neighbourhood $D\subset \R^n$ of $p$ such that, for any $C^r$ (resp. continuous) transition function $\phi$, the $\phi$-regularization $Z_{\e}$ of $Z\big|_D$ can be written as a SP-Problem (according to Definition \ref{spproblem}) such that the reduced system has an invariant manifold $\CS^{\tang}$, which is contained in the slow set $\CS$ and is $C^r$-diffeomorphic (resp. homeomorphically) to $M\cap D$. In addition, the reduced system restricted to $\CS^{\tang}$ is $C^r$-conjugated (resp. to be topologically conjugated) to the tangential sliding vector field $Z_{M\cap D}^{\tang}$.
\end{theorem}

\begin{proof}
First, notice that there exists a neighborhood $D\subset\R^n$ of $p$ and a local coordinate system for which 
$h(x_1, \dots, x_n)=x_n$ and $\eta(x_1, \dots, x_n)=(x_{n-m+1},\dots,x_n)$. In the given coordinates, we get
 \[ \Sigma^{\tang} =  \{p=(x_1, \dots, x_{n-m}, 0, \dots, 0)\} \subset\Sigma.\] 
By denoting $Z^{\pm}=(Z_1^{\pm},\dots ,Z_n^{\pm}),$ the tangential sliding vector field \eqref{eq-campo-tangencial} is written as
\begin{equation}\label{t1k}
Z^{\tang}_M(p)=C(p;\lambda^*(p))=  \frac{(1-\la^*(p))}{2}    Z^+(p)+\frac{(1+\la^*(p))}{2} Z^-(p).
\end{equation}
In what follows, we denote $Z^{\tang}_M=Z^{\tang}=(Z_1^{\tang},\dots ,Z_n^{\tang})$. Note that the trajectories of the tangential sliding vector field satisfy the differential system
\begin{equation}\label{t1f}
\begin{aligned}
\dot x_i  &=Z_i^{\tang}(p)= \frac{(1-\la^*(p))}{2}  Z_{i}^+(p)+\frac{(1+\la^*(p))}{2} Z_{i}^-(p),\quad i=1, 2, \dots,n-m\\
\dot x_j&=Z_j^{\tang}(p)=0,\quad j=n-m+1, \dots, n.
\end{aligned}
\end{equation}
Now, consider the $\phi$--regularization of $Z$, defined by \eqref{regularization} as
\[
Z_{\e}(x)=\frac{(1- \phi_{\e}(x_n))}{2}  Z^+(x)+ \frac{(1+ \phi_{\e}(x_n))}{2}  Z^-(x).
\]
Notice that the trajectories of $Z_{\e}$ satisfy the following differential system
\begin{equation}\label{t1r1}
\dot x_i= \frac{(1- \phi_{\e}(x_n))}{2}  Z_i^+(x)+ \frac{(1+ \phi_{\e}(x_n))}{2}  Z_i^-(x),\quad i=1, \dots, n.
\end{equation}
By applying the change of variables $u=(x_1,x_2,\ldots,x_{n-m})$, $v=(x_{n-m+1}, \dots, x_{n-1}),$  and $w=x_n/\e,$  for $\e>0$ small, the differential system \eqref{t1r1}  is written as the following singular perturbation SP-problem 
\begin{equation}\label{t1sp}
\begin{aligned}
\dot u_i=&\frac{(1-\phi(w))}{2} Z_i^+(u,v, \e w)+\frac{(1+\phi(w))}{2} Z_i^-(u,v, \e w), \quad i=1, \dots, n-m,\\
\dot v_j=&\frac{(1-\phi(w))}{2} Z_i^+(u,v, \e w)+\frac{(1+\phi(w))}{2} Z_i^-(u,v, \e w),\quad j=n-m+1, \dots, n-1,\\
\e \dot w=&\frac{(1-\phi(w))}{2} Z_n^+(u,v, \e w)+\frac{(1+\phi(w))}{2} Z_n^-(u,v, \e w),
\end{aligned}
\end{equation}
where $\e>0$ is the singular perturbation parameter. By taking $\e=0$, we get the reduced problem
\begin{equation}\label{t1sp1}
\begin{aligned}
\dot u_i=&\frac{(1-\phi(w))}{2} Z_i^+(u,v, 0)+\frac{(1+\phi(w))}{2} Z_i^-(u,v, 0), \quad i=1, \dots, n-m,\\
\dot v_j=&\frac{(1-\phi(w))}{2} Z_j^+(u,v, 0)+\frac{(1+\phi(w))}{2} Z_j^-(u,v,0),\quad j=n-m+1, \dots, n-1,\\
0=&\frac{(1-\phi(w))}{2} Z_n^+(u,v, 0)+\frac{(1+\phi(w))}{2} Z_n^-(u,v, 0)=: K(u,v,w),
\end{aligned}
\end{equation}
which describes the dynamics on the ``slow'' timescale $t$ (for standard concepts of singularly perturbed or slow-fast systems see \cite{fenichel,jones}). This dynamics inhabits the {\it slow set} $\CS:=\{(u,v,w):\, K(u,v,w)=0\}$. 

Notice that, in the $(u,v,w)$-coordinate system, $\Sigma^{\tang}=D\cap\{(u,0,0):\, u\in\R^{n-m}\}$. Thus, by taking $w^*(u)=\phi^{-1}(\lambda^*(u,0,0)),$ we have that
\[
\begin{aligned}
K(u,0,w^*(u))=&\frac{(1-\lambda^*(u,0,0))}{2} Z_n^+(u,0, 0)+\frac{(1+\lambda^*(u,0,0))}{2} Z_n^-(u,0, 0)\\
=&Z_n^{\tang}(u,0,0)=0.
\end{aligned}
\]
Hence, $\CS^{\tang}:=\{(u,0,w^*(u))\}$ is a manifold contained in the slow set $\CS$. In addition, for $j=n-m+1, \dots, n-1,$ we have that
\[
\frac{(1-\phi(w^*(u)))}{2} Z_j^+(u,0, 0)+\frac{(1+\phi(w^*(u)))}{2} Z_j^-(u,0,0)=Z_j^{\tang}(u,0,0)=0,
\]
which implies that the flow of the reduced problem \eqref{t1sp1} is invariant over the manifold $\CS^{\tang}$ and is written as
\begin{equation}\label{equacao-sistema-lento}
\begin{aligned}
\dot u_i=&\frac{(1-\lambda^*(u,0,0))}{2} Z_i^+(u,0, 0)+\frac{(1+\lambda^*(u,0,0))}{2} Z_i^-(u,0, 0), \quad i=1, \dots, n-m,\\
\dot v_j=&0,\quad j=n-m+1, \dots, n-1,\\
w=&w^*(u).
\end{aligned}
\end{equation}
Finally, by defining $H:\Sigma^{\tang}\rightarrow \CS^{\tang}$ as
\[
H(p)=(u,0,w^*(u))=(u,0,\phi^{-1}\circ\la^*(p)),\,\, p=(u,0,0),
\]
we can see that the manifold $\CS^{\tang}$ is  $C^r$--diffeomorphic (resp. homeomorphic) to $\Sigma^{\tang}$ provided that $\phi$ is a $C^r$ (resp. continuous) transition function. This completes the proof of the first part of the theorem.

Finally, denote the solution of the differential system \eqref{t1f} starting at $p=(u,0,0)\in\Sigma^{\tang}$ by $t\mapsto x_t(p):=(u(t,p),0,0)$. Similarly, denote the solution of the reduced problem \eqref{equacao-sistema-lento} starting at $H(p)$ by $t\mapsto X_t(H(p))$. Since the first $n-1$ equations of the differential systems \eqref{t1f} and \eqref{equacao-sistema-lento} coincide and do not depend on the last one, we have that
\[
X_t(H(p))= \Big(u(t,p),0, w^*(u(t,p))\Big)=H(x_t(p)).
\]
Consequently, the reduced dynamics restricted to $\CS^{\tang}$  is $C^r$--conjugated (resp. topologically conjugated) to the tangential sliding vector field $Z_{M\cap D}^{\tang}$ provided that $\phi$ is a $C^r$ (resp. continuous) transition function. This completes the proof of the second part of the theorem.
\end{proof}

In Figure \ref{fig regularizacao} we illustrate the double tangential set and the manifold $\CS^{\tang}$. 
\begin{figure}[H]
\begin{center}
\begin{overpic}[width=5.3in]{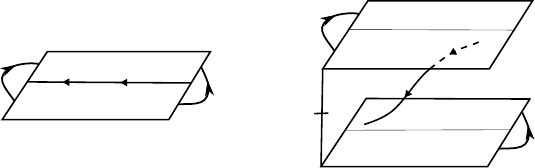}
\put(16,2){(a)} \put(57,9){$0$} \put(57,18){$\varepsilon$} \put(55,0){$- \varepsilon$} \put(20,17){$\Sigma^{\tang}$}
\put(80,15){slow manifold} \put(75,-3){(b)} \put(74,9){$\CS^{\tang}$}
\end{overpic}
\end{center}
\vs
\caption{Regularization. In (a) we get a PSVF with a double tangency set and, in (b), its regularization.}\label{fig regularizacao}
\end{figure}
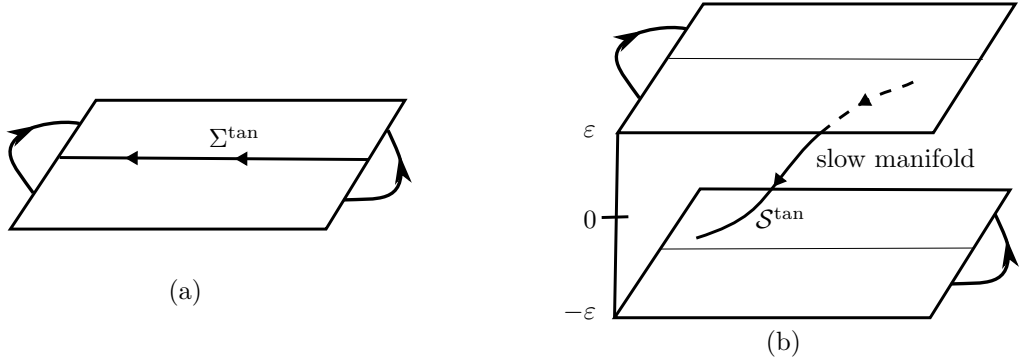

\begin{remark} The invariant manifold $\CS^{\tang}$ of the reduced problem \eqref{t1sp1} is never normally hyperbolic. Indeed,
\[
\dfrac{\p K}{\p w}(u,0,w^*(u))=\phi'(w^*(u))\frac{Z_n^-(u,0, 0)-Z_n^+(u,0, 0)}{2}=0,
\]
for every $u$ such that $(u,0,0)\in D$. Therefore, Fenichel's theory cannot be applied to study the persistence of $\mathcal{S}^{\tang}$ as an invariant manifold of \eqref{t1sp} for $\epsilon>0$ small. Addressing this problem requires the use of blow-up methods (see, for instance, \cite{Dumortier-1977,Seidenberg-1968} and also \cite{NovRond}), which is not in the scope of the present study.
\end{remark}

\begin{remark}
In \cite{Panazzolo-Paulo2017} the notion of solutions of Filippov systems has been extended as being limiting trajectories of some regularization. In this regard, Theorem \ref{teorema-regulariado} shows that the sliding mode provided by Definition \ref{definicao-campo-tagencial}, arising from the Filippov convention, also corresponds to solutions in the context considered by \cite{Panazzolo-Paulo2017}.
\end{remark}

In what follows, we present an example illustrating the equivalence between the tangential vector field and the invariant dynamics on the slow manifold of the $\phi$-regularized vector field.

\begin{example}
Consider the PSVF presented in Example \ref{Exemplo-Dobra-dobra} given by \linebreak $Z=(Z_+, Z_-): U \subset \R^4 \rightarrow \R^4$ where 
\[
\begin{array}{ll}
Z_+  =&(a_1, a_2, a_3, a_4 x_1)\\
Z_-  =&(b_1, b_2, b_3, b_4 x_1).
\end{array}
\]
Recall that the switching manifold is given by $h(x_1, \dots, x_4)=x_4$. By \eqref{t1r1} and con\-si\-de\-ring the notation presented in the proof of Theorem \ref{teorema-regulariado} given by $x_1=v_1, x_2=u_1, x_3=u_2, x_4=w \e$, the regularized vector field is 
\begin{equation}\label{exemplo-reg-dobra-dobraR4}
\begin{array}{ll}
\dot{v_1}  =& (a_1-b_1) \phi(w)+b_1,\\
\dot{u_1}  =&(a_2-b_2) \phi (w)+b_2,\\
\dot{u_2}  =&(a_3-b_3) \phi (w)+b_3,\\
\e \dot{w}  =&v_1 ((a_4-b_4) \phi (w)+b_4).
\end{array}
\end{equation}
In the limit $\e=0$, considering the expression of $\lambda^*(p)$ provided in Theorem \ref{theo1}, we get 
\[
\lambda^*(p) = -1+\dfrac{2\left| a_1\right|}{\left| a_1\right| +\left| b_1\right| }
\]
and the existence of a slow manifold $\mathcal{S}$, given by the restriction $v_1=0$ and $w=\phi^{-1}(\la^*(p))$. Considering the restriction of \eqref{exemplo-reg-dobra-dobraR4} on $\mathcal{S}$, we obtain the reduced problem which, in this case, coincides with the tangential vector field given in \eqref{Exemplo-CampoTangencial-DobraR4}.
\end{example}
\section*{Acknowledgements}

Tiago Carvalho is partially supported by S\~{a}o Paulo Research Foundation (FAPESP) grants  \# 2019/10269-3, \# 2021/12395-6, and \# 2022/02819-6 and by Conselho Nacional de Desenvolvimento Cient\'ifico e Tecnol\'ogico (CNPq), grants 305026/2020-8  and 309378/2023-0.

Douglas Duarte Novaes is partially supported by S\~{a}o Paulo Research Foundation (FAPESP) grants \# 2018/13481-0, \# 2019/10269-3, and \# 2022/09633-5 and by Conselho Nacional de Desenvolvimento Cient\'{i}fico e Tecnol\'{o}gico (CNPq) grant 309110/2021-1. 

Durval Jos\'e Tonon is partially supported by Conselho Nacional de Desenvolvimento Cient\'{i}fico e Tecnol\'{o}gico (CNPq) grant \# 310362/2021-0. 

\bibliographystyle{abbrv}
\bibliography{ReferenciasTangentialSlidingVF}

\begin{thebibliography}{10}

\bibitem{BonLarSea}
C.~Bonet-Rev\'{e}s, J.~Larrosa, and T.~M-Seara.
\newblock Regularization around a generic codimension one fold-fold
  singularity.
\newblock {\em J. Differential Equations}, 265(5):1761--1838, 2018.

\bibitem{BoneSea}
C.~Bonet-Rev\'{e}s and T.~M-Seara.
\newblock Regularization of sliding global bifurcations derived from the local
  fold singularity of {F}ilippov systems.
\newblock {\em Discrete Contin. Dyn. Syst.}, 36(7):3545--3601, 2016.

\bibitem{Brogliato}
B.~Brogliato.
\newblock {\em Nonsmooth Mechanics: Models, Dynamics and Control}.
\newblock Communications and Control Engineering Series. Springer-Verlag
  London, second edition, 1999.

\bibitem{TiagoRonyDurvalLuiz-HIV}
T.~Carvalho, R.~Cristiano, L.~F. Gon\c{c}alves, and D.~J. Tonon.
\newblock Global analysis of the dynamics of a mathematical model to
  intermittent {HIV} treatment.
\newblock {\em Nonlinear Dynamics}, 101:719--739, 2020.

\bibitem{Carvalho2020}
T.~Carvalho, D.~D. Novaes, and L.~F. Gon{\c{c}}alves.
\newblock Sliding shilnikov connection in filippov-type
  predator{\textendash}prey model.
\newblock {\em Nonlinear Dynamics}, 100(3):2973--2987, 2020.

\bibitem{CarTeiTon-CuspFoldZAMP}
T.~Carvalho, M.~A. Teixeira, and D.~J. Tonon.
\newblock Asymptotic stability and bifurcations of {3D} piecewise smooth vector
  fields.
\newblock {\em Zeitschrift f{\"u}r angewandte Mathematik und Physik}, 67(2):31,
  2016.

\bibitem{castro21}
M.~M. Castro, R.~M. Martins, and D.~D. Novaes.
\newblock A note on vishik's normal form.
\newblock {\em Journal of Differential Equations}, 281:442--458, 2021.

\bibitem{Jeffrey-T-sing}
A.~Colombo, M.~di~Bernardo, E.~Fossas, and M.~R. Jeffrey.
\newblock Teixeira singularities in {3D} switched feedback control systems.
\newblock {\em Systems and Control Letters}, 59(10):615--622, 2010.

\bibitem{CarCrisPagTon-PhysicaD-2017}
R.~Cristiano, T.~Carvalho, D.~J. Tonon, and D.~J. Pagano.
\newblock Hopf and homoclinic bifurcations on the sliding vector field of
  switching systems in {$\mathbb{R}^3$}: A case study in power electronics.
\newblock {\em Physica D: Nonlinear Phenomena}, 347:12--20, 2017.

\bibitem{silvamezanova}
P.~R. da~Silva, I.~S. Meza-Sarmiento, and D.~D. Novaes.
\newblock Nonlinear {S}liding of {D}iscontinuous {V}ector {F}ields and
  {S}ingular {P}erturbation.
\newblock {\em Differ. Equ. Dyn. Syst.}, 30(3):675--693, 2022.

\bibitem{diBernardo-livro}
M.~di~Bernardo, C.~J. Budd, A.~R. Champneys, and P.~Kowalczyk.
\newblock {\em Piecewise-smooth Dynamical Systems: Theory and Applications}.
\newblock Number 163 in Applied Mathematical Sciences. Springer-Verlag London,
  first edition, 2008.

\bibitem{Dumortier-1977}
F.~Dumortier.
\newblock Singularities of vector fields on the plane.
\newblock {\em J. Differential Equations}, 23(1):53--106, 1977.

\bibitem{fenichel}
N.~Fenichel.
\newblock Geometric singular perturbation theory for ordinary differential
  equations.
\newblock {\em J. Differential Equations}, 31(1):53--98, 1979.

\bibitem{F}
A.~F. Filippov.
\newblock {\em Differential Equations with Discontinuous Righthand Sides},
  volume~18 of {\em Mathematics and its Applications}.
\newblock Springer Netherlands, first edition, 1988.

\bibitem{RMCGslidingmode}
L.~F. Gon\c{c}alves, D.~S. Rodrigues, P.~F.~A. Mancera, and T.~Carvalho.
\newblock Sliding mode control in a mathematical model to chemoimmunotherapy:
  the occurrence of typical singularities.
\newblock {\em Applied Mathematics and Computation}, 387:124782, 2020.

\bibitem{J-T-T2}
A.~Jacquemard, M.~A. Teixeira, and D.~J. Tonon.
\newblock Piecewise smooth reversible dynamical systems at a two-fold
  singularity.
\newblock {\em International Journal of Bifurcation and Chaos}, 22(8):1250192,
  13, 2012.

\bibitem{Jac-To}
A.~Jacquemard and D.~J. Tonon.
\newblock Coupled systems of non-smooth differential equations.
\newblock {\em Bulletin des Sciences Math\'{e}matiques}, 136(3):239--255, 2012.

\bibitem{jones}
C.~K. R.~T. Jones.
\newblock Geometric singular perturbation theory.
\newblock In {\em Dynamical systems ({M}ontecatini {T}erme, 1994)}, volume 1609
  of {\em Lecture Notes in Math.}, pages 44--118. Springer, Berlin, 1995.

\bibitem{Kousaka}
T.~Kousaka, T.~Kido, T.~Ueta, H.~Kawakami, and M.~Abe.
\newblock Analysis of border-collision bifurcation in a simple circuit.
\newblock In {\em 2000 IEEE International Symposium on Circuits and Systems.
  Emerging Technologies for the 21st Century. Proceedings (IEEE Cat
  No.00CH36353)}, volume~2, pages 481--484, 2000.

\bibitem{Lee-book}
J.~M. Lee.
\newblock {\em Introduction to Smooth Manifolds}.
\newblock Version 3.0. University of Washington, Department of Mathematics,
  2000.

\bibitem{Leine}
R.~Leine and H.~Nijmeijer.
\newblock {\em Dynamics and Bifurcations of Non-Smooth Mechanical Systems},
  volume~18 of {\em Lecture Notes in Applied and Computational Mechanics}.
\newblock Springer-Verlag Berlin Heidelberg, first edition, 2004.

\bibitem{Novaes2022}
D.~Novaes and L.~Silva.
\newblock On the non-existence of isochronous tangential centers in filippov
  vector fields.
\newblock {\em Proceedings of the American Mathematical Society},
  150:5349--5358, 2022.

\bibitem{Novaes15}
D.~D. Novaes and M.~R. Jeffrey.
\newblock Regularization of hidden dynamics in piecewise smooth flows.
\newblock {\em J. Differential Equations}, 259(9):4615--4633, 2015.

\bibitem{NovRond}
D.~D. Novaes and G.~Rond\'{o}n.
\newblock Smoothing of nonsmooth differential systems near regular-tangential
  singularities and boundary limit cycles.
\newblock {\em Nonlinearity}, 34(6):4202--4263, 2021.

\bibitem{Novaes2021}
D.~D. Novaes and L.~A. Silva.
\newblock Lyapunov coefficients for monodromic tangential singularities in
  filippov vector fields.
\newblock {\em Journal of Differential Equations}, 300:565--596, 2021.

\bibitem{Panazzolo-Paulo2017}
D.~Panazzolo and P.~R. da~Silva.
\newblock Regularization of discontinuous foliations: blowing up and sliding
  conditions via {F}enichel theory.
\newblock {\em J. Differential Equations}, 263(12):8362--8390, 2017.

\bibitem{RMCG2019}
D.~S. Rodrigues, P.~F.~A. Mancera, T.~Carvalho, and L.~F. Gon\c{c}alves.
\newblock A mathematical model for chemoimmunotherapy of chronic lymphocytic
  leukemia.
\newblock {\em Applied Mathematics and Computation}, 349:118--133, 2019.

\bibitem{Seidenberg-1968}
A.~Seidenberg.
\newblock Reduction of singularities of the differential equation
  {$A\,dy=B\,dx$}.
\newblock {\em Amer. J. Math.}, 90:248--269, 1968.

\bibitem{Simpson}
D.~J.~W. Simpson.
\newblock {\em Bifurcations in Piecewise-Smooth Continuous Systems}, volume~70
  of {\em World Scientific Series on Nonlinear Science, Series A}.
\newblock World Scientific Publishing, 2010.

\bibitem{Smirnov02}
G.~V. Smirnov.
\newblock {\em Introduction to the theory of differential inclusions},
  volume~41 of {\em Graduate Studies in Mathematics}.
\newblock American Mathematical Society, Providence, RI, 2002.

\bibitem{Regularizacao}
J.~Sotomayor and M.~A. Teixeira.
\newblock Regularization of discontinuous vector fields.
\newblock In {\em International {C}onference on {D}ifferential {E}quations
  ({L}isboa, 1995)}, pages 207--223. World Sci. Publ., River Edge, NJ, 1998.

\bibitem{Teixeira1990}
M.~A. Teixeira.
\newblock Stability conditions for discontinuos vector fields.
\newblock {\em Journal of Differential Equations}, 88(1):15--29, 1990.

\bibitem{TEIXEIRA20121948}
M.~A. Teixeira and P.~R. {da Silva}.
\newblock Regularization and singular perturbation techniques for non-smooth
  systems.
\newblock {\em Physica D: Nonlinear Phenomena}, 241(22):1948--1955, 2012.
\newblock Dynamics and Bifurcations of Nonsmooth Systems.

\bibitem{TG1}
K.~Tirok and U.~Gaedke.
\newblock Regulation of planktonic ciliate dynamics and functional composition
  during spring in lake constance.
\newblock {\em Aquatic Microbial Ecology}, 49(1):87--100, 2007.

\bibitem{TG2}
K.~Tirok and U.~Gaedke.
\newblock Internally driven alternation of functional traits in a multispecies
  predator{\textendash}prey system.
\newblock {\em Ecology}, 91(6):1748--1762, 2010.

\bibitem{Utkin2009}
V.~I. Utkin, J.~Guldner, and J.~Shi.
\newblock {\em Sliding Mode Control in Electro-Mechanical Systems}.
\newblock Automation and Control Engineering. CRC Press, 2009.

\bibitem{V}
S.~M. Vi\v{s}ik.
\newblock Vector fields in the neighborhood of the boundary of a manifold.
\newblock {\em Vestnik Moskov. Univ. Ser. I Mat. Meh.}, 27(1):21--28, 1972.

\end{thebibliography}

\end{document}